\documentclass[11pt]{amsart}
\usepackage{amsmath}
\usepackage{amsfonts}
\usepackage[mathscr]{eucal}
\usepackage{amssymb}
\newtheorem{thm}{Theorem}[section]

\newtheorem{cor}[thm]{Corollary}

\newtheorem{prop}[thm]{Proposition}

\theoremstyle{definition}
\newtheorem{defn}[thm]{Definition}
\newtheorem{remark}[thm]{Remark}

\newcommand{\bb}[1]{\mathbb{#1}}
\newcommand{\cl}[1]{\mathcal{#1}}

\begin{document}

\title[]{Correlation Minimizing Frames in Small Dimensions}

\author[G.~Getzelman]{Grant ~Getzelman}
\address{Department of Mathematics, University of Houston,
Houston, Texas 77204-3476, U.S.A.}
\email{}

\author[N.~L.~Leonhard]{Nicole L.~Leonhard}
\address{Department of Mathematics, University of Houston,
Houston, Texas 77204-3476, U.S.A.}
\email{leonhard@math.uh.edu}

\author[V.~I.~Paulsen]{Vern I.~Paulsen}
\address{Department of Mathematics, University of Houston,
Houston, Texas 77204-3476, U.S.A.}
\email{vern@math.uh.edu}

\thanks{}

\begin{abstract} A uniform tight frame of N vectors for a d dimensional space is {\it correlation minimizing} if among all such frames it is as ``nearly'' orthogonal as possible, i.e., it minimizes the maximal inner product of unequal vectors.  In this paper we  begin to catalog these frames for small dimensions, in particular, d=3.
\end{abstract}

\maketitle

\section{Introduction} 
In their study of the erasures problem, Holmes and the
third author
\cite{HP04} proved that a certain family of frames were optimal for 2
erasures and called these the {\it 2-optimal frames.} Briefly, these are
 uniform Parseval frames for which the minimal angle between any pair of vectors is as large as possible.  Thus, the 2-optimal frames are exactly the frames that are correlation minimizing among all frames in their family of frames. 

They proved that when equiangular
frames exist, then these are always 2-optimal and conversely. But very
few examples are known of 2-optimal frames in the cases when
equiangular frames do not exist. In particular, even in dimension 3,
very few 2-optimal frames are known and even less is known about ``uniqueness'' of such frames, i.e.,  when up to some natural notion of ``equivalence'' of frames a 2-optimal frame is unique.  We
provide additional examples of 2-optimal frames and in some cases we are able to prove uniqueness up to equivalence. 



  Optimal packings of $N$ lines in $R^3$ were studied by Conway, Hardin and
  Sloan  \cite{CHS96} for all values of $N \le 55.$  For some values of $N$
  they were able to give closed form descriptions of these sets of
  lines, along with proofs that they were indeed optimal packings,
  while for many values of $N$ they were only able to give numerical
  approximations to these optimal packings. Holmes and Paulsen
  \cite{HP04} also did numerical experiments that attempted to compute
  the minimum angle between vectors for 2-optimal frames of $N$
  vectors in $\bb R^3.$ Their computations showed that for some
  values of $N$ the minimum angle between vectors for 2-optimal frames
  appeared to be identical to the angle determined by \cite{CHS96} for
  optimal line packings. This lead
  them to conjecture that for these values on $N$, one could
  obtain uniform tight frames by choosing a unit vector from each line in
  the optimal line packing.

In this paper we find some of these 2-optimal frames in $\bb R^3$ for some values of $N$ and prove that for some of the values of $N$ where these
two angles were shown to numerically agree, that one does indeed
obtain tight frames, which are necessarily 2-optimal, by choosing unit
vectors from the optimal line packing. Also, in some cases earlier work only provided the
Grammian matrices of these 2-optimal frames, but for the convenience of
the reader who might be interested in experimenting with these frames,
we also produce the actual frames and some geometric descriptions of the
sets of vectors.



In the next section we review the necessary background information
needed and then in the final section we prove the results claimed
above and produce the frames.

\section{Background} 
We briefly review the key concepts from frame theory and the paper \cite{HP04} that we shall need.

\subsection{Frames}

A family $\cl F = \{f_i \}_{i\in I}$ of elements in a (real or complex) Hilbert space
$\cl H$ is called a \emph{frame} for $\b H$ if there are constants $0<A\le B < \infty,$  called the \emph{lower and upper frame bounds}, respectively
so that for all $f\in \cl H$
\begin{equation}
A\|f\|^2 \le \sum_{i\in I}|\langle f,f_i \rangle |^2 
\le B\|f\|^2.
\end{equation}

In general, a frame can have more vectors than the dimension of the Hilbert space and, in the case that the space is finite dimensional, we call
\[ \frac{card(I)}{dim(\cl H)} \]
the {\em redundancy} of the frame.
  
  If $A=B$, then $\cl F$ is called a \emph{tight frame} and when $A=B=1$, $\cl F$ is called a \emph{Parseval frame}. Thus, if $\cl F= \{ f_i \}_{i \in I}$ is a tight frame with constant $A$ then $\{ \frac{1}{\sqrt{A}} f_i \}_{i \in I}$ is a Parseval frame.

While many authors prefer to work with unit norm tight frames, we will mainly consider uniform Parseval frames. As we see above this is just a matter of scaling.

If $\cl F = \{f_i \}_{i\in I}$ is a frame for $\cl H$,
the \emph {analysis operator} is the bounded linear operator $V:\cl H \rightarrow {\ell}_{2}(I)$ given by $V(x)_i= \langle x,f_i \rangle$ for all $i\in I$. 
The adjoint of the analysis operator, $V^*: \ell_2(I) \to \cl H$ is defined by the formula, $V^*(e_i) = f_i,$ where $e_i$ is the vector that is $1$ in the $i$-th entry and 0 elsewhere. Hence, 
\[V^*Vx = \sum_{i \in I} \langle x, f_i \rangle f_i .\]

In particular, $\cl F$ is a Parseval frame if and only if $V$ is an isometry  and this is if and only if 
$V^*V = I_{\cl H}.$

Thus,  $\cl F= \{ f_i \}_{i \in I}$ is a Parseval frame if and only if we have that
\[ h = \sum_{i \in I} \langle h, f_i \rangle f_i, \, \forall h \in \cl H.\]
More generally,
if $\cl F= \{ f_i \}_{i \in I}$ is a tight frame for a Hilbert space $\cl H$ with constant $A$ then
\[ h = \frac{1}{\sqrt{A}} \sum_{i \in I} \langle h, f_i \rangle f_i, \, \forall h \in \cl H.\]
This is known as the {\it sampling and reconstruction formula.}

On the other hand the analysis operator $V$ is an isometry if and only if
the \emph{Grammian matrix}  $VV^*=\left( \langle f_j,f_i\rangle \right)$ is a projection with rank equal to $dim(\cl H),$ which gives another characterization of Parseval frames.

If $\cl F = \{ f_1,...,f_N\}$ is a Parseval frame  for a $d$-dimensional Hilbert space, then to shorten terminology we shall simply call $\cl F$ a {\it  $(N,d)$-frame.}  If $\cl F$ is a uniform $(N,d)$-frame with analysis operator $V$, then $V^*V= I_d$ the $d \times d$ identity matrix and
\[ d= rank(VV^*) = Tr(VV^*) =  \sum_{i=1}^N \|f_i\|^2 = N \|f_k\|^2,\]
for any $k.$ Thus, for a $(N,d)$-frame,
\[ \frac{1}{\|f_k\|^2} = \frac{N}{d},\]
is the redundancy.

For this reason, when $\cl F= \{ f_i \}_{i \in I}$ is a uniform Parseval frame for an infinite dimensional Hilbert space, we still call $\frac{1}{\|f_k\|^2}$ the {\it frame redundancy.}

Following the notation of \cite{HP04}, we let $\cl E_1(N,d)$ denote the set of all uniform $(N,d)$-frames  for $\bb R^d$(respectively, $\bb C^d$). This notation comes from a result of Casazza and Kovalecevic \cite{CK} that the uniform Parseval frames are in a certain sense optimal for the 1-erasure problem.

Given an indexed set $\cl S= \{ v_i \}_{i \in I}$ of non-zero vectors, their  \emph{ maximum correlation,} denoted  $M_\infty(\cl S)$, is defined as
\begin{equation}
M_\infty(\cl S) = \sup \{|\langle \frac{v_k}{\|v_k\|}, \frac{v_l}{\|v_l\|} \rangle|:  k,l \in I, \, k \ne l \}
\end{equation}
so that
\[ \Theta(\cl S) := arccos(M_\infty(\cl S))\]
is the infimum of the angles between pairs of vectors.
Thus, 
$0 \le M_\infty (\cl S) \le 1$  and $M_\infty(\cl S) = 0$ if and only
if $\Theta(\cl S) = \frac{\pi}{2}$ if and only if $\cl S$ is an orthogonal set. Thus,  smaller maximum correlation means a set is more nearly orthogonal. If $M_\infty(\cl S) =1$ and the supremum is attained, then any two vectors where the supremum is attained are parallel. So larger maximum correlation indicates that the set contains vectors that are more nearly parallel.

Given $\cl F= \{ f_1,...f_N \} \in \cl E_1(N,d)$ we have
\[ M_\infty(\cl F) = \frac{N}{d} \cdot \max \{ | \langle f_k, f_l \rangle| : k \ne l \}.\]
Note that the factor $\frac{N}{d}$ is included because $\cl F \in \cl
E_1(N,d)$ implies that each vector in $\cl F$ has norm $\sqrt{d/N}.$

We set
\begin{equation}
C(N,d) = \inf \{ M_\infty(\cl F): \cl F \in \cl E_1(N,d) \}.
\end{equation}
and
\[ \Theta(N,d) = arccos(C(N,d)).\]
Note that since $arccos$ is a decreasing function,
\[ \Theta(N,d) = \sup \{ \Theta(\cl F): \cl F \in \cl E_1(N,d) \}.\]
We shall call $C(N,d)$ the {\it correlation constant for frames in $\cl E_1(N,d)$}
and so $\Theta(N,d)$ is the {\it maximum angle between vectors for frames in $\cl E_1(N,d)$.}

These constants were introduced in \cite{HP04} where it was proven
that $C(N,d)$ was always attained, i.e., the infimum is actually a
minimum, and any uniform Parseval frame where it is attained was
called {\it 2-optimal}. This terminology arose from the result from \cite{HP04} that such frames were optimal for the 2-erasures problem.

Perhaps a more descriptive terminology would have been to call  a uniform Parseval frame where this minimum correlation is attained a {\it correlation minimizing} uniform Parseval frame or {\it angle maximizing} uniform Parseval frame. 
In which case the result from \cite{HP04} is that a uniform Parseval frame is correlation minimizing if and only if it is optimal for the 2-erasure problem.

In any case, since we will not be considering the erasures problem in
this paper, we prefer to use the term correlation minimizing.  Thus, a uniform Parseval frame is {\bf correlation minimizing} if and only if it is in the set 
\[ \cl E_2(N,d)= \{ \cl F \in E_1(N,d):  M_\infty(\cl F) = C(N,d) \}\]
and a uniform Parseval frame is correlation minimizing if and only if
it is 2-optimal. In \cite{BH14} the correlation minimizing frames are
called {\it Grassmannian Parseval frames}.

A frame is \emph{equiangular} if for some $\alpha, | \langle f_i , f_j \rangle | = \alpha$ for all i$ \neq$ j.  

\begin{thm}\cite{HP04}
Let $N \geq d$, and let $\cl F \in \cl E_1(N,d).$ Then
\begin{equation}
M_\infty(\cl F) \geq \sqrt{\frac{N - d}{d(N-1)}},
\end{equation}
and equality holds iff $\cl F$ is equiangular. \\
If there exists an equiangular frame $\cl F \in \cl E_1(N,d),$ then it is 2-optimal, i.e., correlation minimizing and in this case every frame in $\cl E_2(N,d)$ is equiangular.\\
Futhermore,  if $N > \frac{d(d+1)}{2}$ in the real case and $N > d^2$ in the complex case, then there is no equiangular frame in $\cl E_1(N,d)$ and equality cannot hold in the above equation.
\end{thm} 

Thus, we see that
\[ C(N,d) \ge \sqrt{\frac{N-d}{d(N-1)}},\]
with equality if and only if there exists an equiangular frame. The quantity appearing on the right hand side of the above equation is known as the {\it Welch bound.}

Since $C(N,d)$ is the actual lower bound for the maximum correlation of uniform Parseval frames and, consequently, for unit norm tight frames, it should be an important constant to compute. But surprisingly very little is known about it other than the numerical estimates in \cite{HP04}.


The goal of this paper is to find more examples of correlation minimizing uniform Parseval frames and to determine if they are  unique, modulo an equivalence relation, which we now explain.

\begin{defn} 
Frames $\cl F$ = $\{f_i\}_{i=1}^n$ and  $\cl G$ = $\{g_i\}_{i=1}^n$, are \emph{type I equivalent} if there exists a unitary (orthogonal matrix, in the real case) $U$ such that $g_i = Uf_i$ for all $i$. 
\end{defn} 

\begin{thm}\cite{HP04}\label{typeI}
If $V$ and $W$ are the analysis operators for $\cl F$ and $\cl G$, respectively, then the following are equivalent\\
1. $\cl F$ and $\cl G$ are \emph {type I equivalent} \\
2. there exists a unitary(respectively, orthogonal matrix) $U$ such that $V = WU$ \\
3. $VV^* = WW^*$.  
\end{thm}

In \cite{HP04} it is shown that
there is a one-to-one correspondence between $N \times N$ rank $d$ projections and type I equivalence classes of  uniform $(N,d)$-frames.  

\begin{defn}Frames $\cl F= \{ f_i \}_{i=1}^n$ and $\cl G= \{ g_i \}_{i=1}^n$ are \emph{type II equivalent} if they are a
  permutation of the same set of vectors and they are \emph{type III
    equivalent} if there exist numbers $\{\lambda_i\}_{i=1}^n$ of
  modulus such that $f_i = \lambda_i g_i.$ Thus, in the real case if they differ by multiplication by $\pm 1$.  Two frames are \emph{equivalent} if they belong to the same equivalence class in the equivalence relation generated by these three equivalence relations. 
\end{defn}

\begin{thm}\cite{HP04}\label{equivalence}
 If $\cl F$ and $\cl G$ are $(N,d)$-frames with analysis operators $V$ and $W$, respectively, then  they are equivalent if and only if $UVV^*U^* = WW^*$ for some $N \times N$ unitary $U$ that is the product of a permutation matrix and a diagonal matrix with entries of modulus 1($\pm 1$, in the real case). 
\end{thm}




\section{Grassmannian and Correlation Minimizing Frames in $\bb R^3$} 

The {\it  line packing problem} is the problem of packing $N$ lines in
$\bb R^d$ so that the minimal angle between any two of them is as
large as possible.  Any solution to this problem is called a {\it
  Grassmannian line packing.} Given a Grassmannian line packing with
$N \ge d$ if we choose one unit vector from each line, then this set
of vectors always yields a frame for $\bb R^d.$  Any frame obtained
this way is called a  {\it Grassmannian frame} by \cite{SH03}.

If a Grassmannian frame happens to be a tight frame, then after scaling the vectors by $\sqrt{d/N}$ we would obtain a uniform $(N,d)$-frame that is necessarily correlation minimizing.
Thus,  Grassmannian frames yield 2-optimal frames by scaling if and only if they are tight frames. 
 
A geometric approach to solving the line packing problem and list of best-known packings is posted on \cite{SL}. 
 Conway, Hardin, and Sloane \cite{CHS96} find the Grassmannina line packings of $N$ lines through the origin in $\bb{R}^3$, describe the packings and compute this minimal angle for $2 \leq N \leq 55$.  So a natural question that we shall study below is whether or not the Grassmannian frames arising from these Grassmannian line packings are tight. The numerical experiments of \cite{HP04} indicates that the answer should be ``yes'' for some values of $N$ and ``no'' for other values.


In \cite{HP04} the uniform $(N,2)$-frames that are correlation minimizing were constructed, it was shown that these frames form a single equivalence class, and that these are Grassmannian.

First note that whenever $N=d$ then any Parseval frame must be an orthonormal basis, since the frame operator $V$ will be an isometry from $\bb R^d$ to $\bb R^d$ and hence will be an orthogonal matrix.  Hence, the rows of $V$ will be an orthonormal set. Moreover, every orthonormal basis is type I equivalent. Thus, there is a unique equivalence class of $(d,d)$-frames and an orthonormal basis is clearly correlation minimizing and Grassmannian.

So the first interesting case is the $(4,3)$-frames.
For this a theorem is useful.

\begin{thm}\label{I-G}
Let $\cl F$ = $\{f_i\}_{i=1}^N$ be a correlation minimizing $(N,d)$-frame with Grammian matrix $G$.  Then $I_N -G$ is the Grammian matrix of a correlation minimizing $(N, N-d)$-frame.
Moreover, there is a one-to-one correspondence between equivalence classes of correlation minimizing $(N,d)$-frames and equivalence classes of correlation minimizing $(N, N-d)$-frames.
\end{thm}

\begin{proof} Let $\cl F$ = $\{f_i\}_{i=1}^N$ be any uniform $\left(N,d\right)$-frame with Grammian matrix $G$.  Then $G$ is a rank $d$ projection all of whose diagonal entries are equal to $\frac{d}{N}.$ Hence, $I_N - G$ is a rank $N-d$ projection all of whose diagonal entries are $\frac{N-d}{N}.$  Hence, if we let  $WW^*= I_N -G$ be any factorization, then by the results of the last section the rows of $W$(or their complex conjugates in the complex case) form a uniform $(N, N-d)$-frame whose Grammian is $I_N - G.$ This frame is not uniquely determined by $I_N -G,$ since many factorizations are possible, but it is unique up to type I equivalence by Theorem~\ref{typeI}. 

So choose one such uniform $(N, N-d)$-frame and denote it by $\cl F^{\perp}.$
Now if $\cl F_i$ $i=1,2$  are any uniform $(N,d)$-frames with Grammians $G_i$ $i=1,2$ and we let $\cl F_i^{\perp}$ $i=1,2$ be $(N,N-d)$-frames with Grammians $I_N - G_i$ obtained as above, then, by applying  Theorem~\ref{equivalence}, we see that $\cl F_1$ and $\cl F_2$ are equivalent if and only if $\cl F_1^{\perp}$ and $\cl F_2^{\perp}$ are equivalent.

Finally, since the maximum correlation of a frame is really just the maximum off-diagonal entry of its Grammian(appropriately scaled), we see that
\[ M_{\infty}(\cl F^{\perp}) = \frac{d}{N-d} M_{\infty}(\cl F),\]
and so whenever $\cl F$ is a correlation minimizing uniform $(N,d)$-frame that $\cl F^{\perp}$ is a correlation minimizing uniform $(N,N-d)$-frame.
\end{proof}

\begin{remark}  The above proof also shows that
\[ C(N, N-d) = \frac{d}{N-d}C(N,d).\]
\end{remark}
\medskip

Let $J_N$ denote the $N \times N$ matrix of all $1$'s.

\begin{cor} Up to equivalence there is a unique correlation minimizing $(N,1)$-frame and a unique correlation minimizing $(N, N-1)$-frame. These equivalence classes are represented by the uniform frames with Grammians $\frac{1}{N} J_N$ and $I_N - \frac{1}{N} J_N,$ respectively. Moreover, both these frames are equiangular and so these  frames are also Grassmannian.
\end{cor}
\begin{proof} To obtain a uniform $(N,1)$-frame one must choose $N$ numbers of modulus $1/\sqrt{N}.$ But these are all  equivalent to choosing the number
$1/\sqrt{N}$ $N$-times.  Thus, up to equivalence there is only one uniform $(N,1)$-frame and it has Grammian $\frac{1}{N} J_N.$

Hence, by the above theorem, up to equivalence there is only one $(N, N-1)$-frame and it has Grammian given by $I_N - \frac{1}{N} J_N.$

All these frames are equiangular since all the off-diagonal entries in their Grammians are of constant modulus $\frac{1}{N}.$
\end{proof}

We can now give a concrete description of one representative of this equivalence class of frames in the case $N=4,$ $d=N-1=3.$

\begin{thm} \label{(4,3)}The lines generated by opposite vertices of the inscribed cube in the sphere centered at the origin is the optimal packing of 4 lines in 3-space. If we take the sphere of radius $\frac{\sqrt{3}}{2}$ centered at the origin and consider the 8 vectors determined by the vertices of this cube, then any set of 4 of these vectors that are not collinear yields a  correlation minimizing, equiangular $\left(4,3\right)$-frame.
In particular, one correlation minimizing, equiangular $(4,3)$-frame is given by
\[ (+\frac{1}{2}, +\frac{1}{2}, +\frac{1}{2}), 
(-\frac{1}{2}, -\frac{1}{2}, +\frac{1}{2}), (-\frac{1}{2},
+\frac{1}{2}, -\frac{1}{2}), (+\frac{1}{2}, -\frac{1}{2},
-\frac{1}{2}) \]
and every other correlation minimizing (4,3)-frame is equivalent to
this frame.
\end{thm}
\begin{proof} Let
\[V= \left( \begin{matrix} {f_1}^* \\ {f_2}^* \\ {f_3}^* \\ {f_4}^* \end{matrix} \right) = \left( \begin{matrix}
+\frac{1}{2}& +\frac{1}{2}& +\frac{1}{2} \\
-\frac{1}{2}& -\frac{1}{2}& +\frac{1}{2} \\
-\frac{1}{2}& +\frac{1}{2}& -\frac{1}{2} \\
+\frac{1}{2}& -\frac{1}{2}& -\frac{1}{2} \end{matrix}\right).\]
Computing the Grammian yields $G=VV^*= I_4 - \frac{1}{4}J_4$ so that this is one representative of the unique correlation minimizing $(4,3)$-frame.

The remaining claims are now straightforward to verify.
\end{proof}

\begin{thm} \label{(5,3)} A correlation minimizing $(5,3)$-frame is given by the vectors:
\begin{multline*}
c (1, 0,  0), \, \, c ( \frac{-1-\sqrt{5}}{6},  \frac{15 - \sqrt{5}}{a},  0) \\
 c (\frac{1-\sqrt{5}}{6},   \frac{-5-3\sqrt{5}}{a},  \frac{150-30\sqrt{5}}{ab}),  
c (\frac{-1+\sqrt{5}}{6},   \frac{5-3\sqrt{5}}{a},   \frac{-60\sqrt{5}}{ab}) \\
c (\frac{1 + \sqrt{5}}{6},  \frac{4\sqrt{5}}{a},  \frac{150-30\sqrt{5}}{ab}),
\end{multline*}
where  $a = \sqrt{18\left(15-\sqrt{5}\right)}$, 
  $b = \sqrt{150-30\sqrt{5}}$, and $c = \sqrt{3/5}.$ Every other correlation minimizing $(5,3)$-frame is equivalent to this frame.
\end{thm}

\begin{proof}
From \cite{HP04} we have that the correlation minimizing $\left( 5,2\right)$-frame is unique up to equivalence and one representative is given by the vectors 
\[\{ \big( cos\left(\frac{\pi k}{5}\right), sin\left(\frac{\pi k}{5}\right) \big)  : k=1,2,3,4,5 \}.\] 
Thus, by Theorem~\ref{I-G} the correlation minimizing $(5,3)$-frame will be unique up to equivalence and a representative Grammian will be given by $G= I_5 - G_{(5,2)},$ where
$G_{(5,2)}$ is the Grammian of the above vectors. 

Computing this Grammian yields,
\[G =\left( \begin{smallmatrix} 3/5& \frac{2}{5}cos(\pi/5)& \frac{2}{5}cos(2\pi/5) & \frac{2}{5}cos(3\pi/5) & \frac{2}{5}cos(4\pi/5) \\ \frac{2}{5}cos(\pi/5) & 3/5 & \frac{2}{5}cos(\pi/5) & \frac{2}{5}cos(2\pi/5) & frac{2}{5}cos(3\pi/5)\\ \frac{2}{5}cos(2\pi/5) & \frac{2}{5}cos(\pi/5) & 3/5 & \frac{2}{5}cos(\pi/5) & \frac{2}{5}cos(2\pi/5)\\ \frac{2}{5}cos(3\pi/5) & \frac{2}{5}cos(2\pi/5) & \frac{2}{5}cos(\pi/5)& 3/5& \frac{2}{5}cos(\pi/5)\\ \frac{2}{5}cos(4\pi/5) & \frac{2}{5}cos(3\pi/5) & \frac{2}{5}cos(2\pi/5) & \frac{2}{5}cos(\pi/5) & 3/5  \end{smallmatrix} \right) .\]
  This can be factored as $G=\frac{3}{5}UU^*$ where 
\[U= \begin{pmatrix}
  1 & 0 & 0 \\
   \frac{-1-\sqrt{5}}{6} & \frac{15 - \sqrt{5}}{a} & 0 \\
   \frac{1-\sqrt{5}}{6} &  \frac{-5-3\sqrt{5}}{a} &  \frac{150-30\sqrt{5}}{ab} \\
   \frac{-1+\sqrt{5}}{6} &  \frac{5-3\sqrt{5}}{a} &  \frac{-60\sqrt{5}}{ab} \\
  \frac{1 + \sqrt{5}}{6} & \frac{4\sqrt{5}}{a} &  \frac{150-30\sqrt{5}}{ab} \\
  \end{pmatrix},\]
with $a$ and $b$ as above. 
  \end{proof}
 \begin{cor} The Grassmannian frame of 5 vectors in $\bb R^3$ is not a tight frame.
\end{cor}
\begin{proof}
By inspection the largest off diagonal entry of the above $G$  is $\frac{2}{5}cos(\pi/5),$ and the smallest angle produced by the vectors of this Grammian is equal to $\arccos(\frac{2}{3} cos(\pi/5))$ which is approximately $57.361$ degrees. This  is not equal to the angle of the optimal packing of 5 lines  found in \cite{CHS96}.
Thus, if we take one unit vector from each of the 5 lines corresponding to the optimal packing of 5 lines through the origin in $\bb R^3,$ then this set of vectors can not be a tight frame since its correlation is smaller.
\end{proof}
 
Thus, the correlation minimizing $(5,3)$-frame is an example that is not obtained via the optimal line packing.  In the language of \cite{SH03} the correlation minimizing $(5,3)$-frame is not a Grassmannian frame. In the language of \cite{BH14} the Grassmannian Parseval frame of 5 vectors in $\bb R^3$ is not a Grassmannian frame.

In \cite{BK06} it was shown that the correlation minimizing $(6,3)$-frame is equiangular, that it is unique up to equivalence and its Grammian was given. Below we give a geometric description of the set of vectors for one representative of this equivalence class and give the vectors explicitly.

\begin{thm} \label{(6,3)} The $6$ vertices, that lie in the upper half plane of an icosahedron centered at the origin and symmetric about the $xy$-plane, form a correlation minimizing $(6,3)$-frame.  
Set $\alpha = \frac{1}{\sqrt{5}}$, then these are the vectors given by: \\
$f_1 = \frac{1}{\sqrt{2}} \left(0,0,1\right)$, \\
$f_2 = \frac{1}{\sqrt{2}}\left(\sqrt{1 - \alpha^2},0,\alpha \right)$,\\
$f_3 = \frac{1}{\sqrt{2}}\left(\alpha \sqrt{\frac{1 - \alpha}{1 + \alpha}}, \sqrt{\frac{(1 + 2\alpha)(1 - \alpha)}{1 + \alpha}}, \alpha \right)$, \\
$f_4 = \frac{1}{\sqrt{2}}\left( \alpha \sqrt{\frac{1 - \alpha}{1 + \alpha}}, -\sqrt{\frac{(1 + 2\alpha)(1 - \alpha)}{1 + \alpha}}, \alpha \right)$, \\
$f_5 = \frac{1}{\sqrt{2}}\left(-\alpha \sqrt{\frac{1 - \alpha}{1 + \alpha}}, \sqrt{\frac{(1 - 2\alpha)(1 + \alpha)}{1 - \alpha}}, \alpha \right)$, \\
$f_6 =\frac{1}{\sqrt{2}}\left( -\alpha \sqrt{\frac{1 + \alpha}{1 - \alpha}}, -\sqrt{\frac{(1 - 2\alpha)(1 + \alpha)}{1 - \alpha}}, \alpha\right)$. \\
Every other correlation minimizing $(6,3)$-frame is equivalent to this frame.
\end{thm}
\begin{proof}
From \cite{BK06} we have the vectors defined above.  For $k \neq l$, we compute $|\langle f_k, f_l \rangle| = \frac{1}{\sqrt{5}}$. Thus, this set of vectors is equiangular and and each vector has norm $\frac{\sqrt{3}}{\sqrt{6}}$ so these must be a correlation minimizing $(6,3)$-frame.

In \cite{CHS96} it is observed that the $6$ lines obtained by taking  antipodal pairs of points on an icosahedron are equiangular.
\end{proof}

Before constructing a correlation minimizing $(7,3)$-frame a little proposition will be useful.

\begin{prop} If $\{ f_1, \ldots, f_N \}$ is a uniform $(N,d)$-frame and $\{g_1, \ldots, g_M \}$ is a uniform $(M,d)$-frame, then $\{ af_1, \ldots, af_N, bg_1, \ldots, bg_M \}$ is a uniform $(M+N,d)$-frame, where $a= \sqrt{N/(N+M)}$ and $b= \sqrt{M/(N+M)} $.
\end{prop}
\begin{proof} Since $\|f_i\|= \sqrt{d/N}$ and $\|g_j\|= \sqrt{d/M}$ we have that
$\|a f_i\| = \|bg_j\| = \sqrt{d/(N+M)},$ so this set of vectors is uniform in norm.
Finally, for any vector $x \in \bb R^d,$ we have that
\[ \sum_{i=1}^N | \langle x, a f_i \rangle|^2 + \sum_{j=1}^M |\langle x, bg_j \rangle |^2 = a^2\|x\|^2 + b^2 \|x\|^2 = \|x\|^2,\]
so the Parseval condition is met.
\end{proof}

\begin{thm} \label{(7,3)} Let $\{ f_1, f_2, f_3, f_4 \}$ be the correlation minimizing $(4,3)$-frame of Theorem~\ref{(4,3)} and let $\{ e_1, e_2,e_3\}$ be the standard orthonormal basis for $\bb R^3$, then $\{ \sqrt{4/7}\, f_1, \sqrt{4/7}\, f_2, \sqrt{4/7}\, f_3, \sqrt{4/7}\, f_4, \sqrt{3/7}\, e_1, \sqrt{3/7}\, e_2, \sqrt{3/7}\, e_3 \}$ is a correlation minimizing $(7,3)$-frame.
\end{thm}
\begin{proof} By the above proposition, this set of vectors is a uniform $(7,3)$-frame.  The inner products of pairs of these unequal vectors take on the values,
$\{ 0, \, \pm \frac{\sqrt{3}}{7}, \, \pm \frac{1}{7} \},$ so that for this frame,
$M_{\infty}(\cl F) = \frac{7}{3} \frac{\sqrt{3}}{7} = \frac{\sqrt{3}}{3}.$
 Since $cos^{-1}\left( \frac{\sqrt{3}}{3}\right)$  corresponds to the minimum angle for the Rhombic Dodecahedron \cite{RD}, which is an optimal line packing angle for 7 lines in 3 space found by \cite{CHS96}, this uniform Parseval frame must be correlation minimizing.
\end{proof}

Since the correlation minimizing (7,3)-frame corresponds to an optimal
line packing, every correlation minimizing (7,3)-frame would yield an
optimal line packing. But
we do not know if every correlation minimizing (7,3)-frame is
equivalent to this frame.
In \cite{CHS96}, they remark that the optimal packing of 7 lines in 3
space appears to be unique, but do not supply a proof. A related, and
possibly easier,
problem would be to decide if every optimal line packing of 7 lines in
3 space yields a tight frame.

The optimal line packing for $10$ lines in $\bb R^3$ is given numerically on Sloane's web site \cite{SL}.  In \cite{CHS96}, it was determined that there are infinitely many solutions to this optimal line packing problem.  This occurs because the axial line can rattle, that is, the vectors can move freely over a small range of angles without affecting the minimum angle.  

\begin{thm}\label{(10,3)}  The optimal line packing for $10$ lines in $\bb R^3$ comprised of 2 axis vectors and the set of 8 vectors that are not collinear from the scaled hexakis bi-antiprism, given by\\
	\begin{multline*}
	\left(1,0,0\right), \left(0,-1,0\right), \\
	\left(\pm \frac{ \sqrt{3}}{2},\frac{1}{2},0 \right),
	\left( \beta, 0 ,\pm \beta\sqrt{\sqrt{3}-1} \right)\\
	\left( \frac{\beta}{2}, \beta\frac{\sqrt{3}}{2},\beta \sqrt{\sqrt{3}-1}\right),
\left( -\frac{\beta}{2},-\beta\frac{\sqrt{3}}{2},\beta\sqrt{\sqrt{3}-1}\right),\\
 \left(\frac{\beta}{2},-\beta\frac{\sqrt{3}}{2},-\beta\sqrt{\sqrt{3}-1}\right),
 \left(-\frac{\beta}{2},\beta\frac{\sqrt{3}}{2},-\beta\sqrt{\sqrt{3}-1}\right).
 \end{multline*}
 where $\beta=3^{-\frac{1}{4}}$, is not a tight frame.  Moreover, there does not exist a rattle of the axis that will create it a tight frame.
  \end{thm}
\begin{proof} 
	
First, we construct the hexakis bi-antiprim 
by taking two hexagonal antiprisms and joining them at the base.  To create the first half, shift the coordinates for the hexigonal antiprism from \cite{DC}.  For the second half, we use a shift and rotation of the same coordinates from \cite{DC}.  Now, we join them at the base to complete the construction.  From the set of 18 unique scaled vectors in the construction we consider 2 axis together with the set of 8 vectors that are not collinear. Set $\beta=3^{-\frac{1}{4}}$ and define \\
$V= \left(\begin{array}{ccc}

1 & 0& 0\\
 0 &-1 & 0\\
-\frac{ \sqrt{3}}{2}&\frac{1}{2} & 0 \\
\frac{ \sqrt{3}}{2}&\frac{1}{2} & 0 \\
 \beta& 0 & \beta\sqrt{\sqrt{3}-1} \\
 \beta& 0 &-\beta\sqrt{\sqrt{3}-1} \\
\frac{\beta}{2}& \beta\frac{\sqrt{3}}{2}&\beta \sqrt{\sqrt{3}-1}\\
-\frac{\beta}{2}&-\beta\frac{\sqrt{3}}{2}&\beta\sqrt{\sqrt{3}-1}\\
\frac{\beta}{2}&-\beta\frac{\sqrt{3}}{2}&-\beta\sqrt{\sqrt{3}-1}\\
-\frac{\beta}{2}&\beta\frac{\sqrt{3}}{2}&-\beta\sqrt{\sqrt{3}-1}\\
\end{array}\right)$

Recall that a set of vectors forms a uniform tight frame if and only
if they are of equal norm and when they are entered as the rows of a matrix, then that matrix is
a multiple of an isometry. Moreover, to be a multiple of an isometry,
the columns of the matrix must be orthogonal and of equal norm.

The rows of $V$ are unit norm. By inspection we see the columns are
orthogonal. However, the columns of $V$ do not have equal norm.
Hence, no multiple of $V$ is an isometry and so the rows are not a
tight frame.  

Now, we will consider the case where the axial lines "rattle" to try to gain equality in the norm of the columns. 
  Consider the first row as $v_1=\left(a_1,b_1,c_1\right)$ and the
  second as $v_2=\left(a_2,b_2,c_2\right)$.  Since the vectors
  comprising the last eight entries of $V$ are orthogonal, to keep the
  columns orthogonal we will need the vectors $(a_1,a_2), (b_1, b_2)$
  and $(c_1,c_2)$ to be orthogonal. Since this is three vectors in $\bb R^2$, one of them must be zero.  The norm of the first column is the largest so $a_1=a_2=0$. The rows must be unit norm so $b_1^2+c_1^2=1$ and $b_2^2+c_2^2=1$.
We still need the norms of the three columns to be equal.  Thus, we get the system of equations. 
\begin{equation*}
\left\lbrace\begin{matrix}
b_1^2+c_1^2=1 \\b_2^2+c_2^2=1\\b_1^2+b_2^2=1\\c_1^2+c_2^2=3\sqrt{3}-\frac{9}{2}\\
\end{matrix}. \right.
\end{equation*}
By subtracting the third equation from the first we see that
$c_1^2=b_2^2$.  Plugging into equation 2 we get $c_1^2+c_2^2=1$, which
contradicts the fourth equation. Therefore, there is no choice of
vectors that can make V a multiple of an isometry. 
\end{proof}

\begin{thm} A correlation minimizing $\left(12,3\right)$-frame is
  given by scaling the set vertices of the rhombicuboctahedron through
  the origin to be vectors of length $\frac{1}{2}$. This frame is
  given by all non-collinear permutations of the vectors $\frac{1}{2} \left(\pm \frac{1}{\sqrt{2\sqrt{2}+5}} ,\pm \frac{1}{\sqrt{2\sqrt{2}+5}}, \pm \frac{(1+\sqrt{2})}{\sqrt{2\sqrt{2}+5}} \right).$
\end{thm}
\begin{proof}
	In \cite{CHS96}, the optimal line packing of 12 lines in $\bb
        R^3$ is a rhombicuboctahedron.  Define V such that the rows
        are the vectors of the rhombicuboctahedron in \cite{RD} . So,
        $V= \frac{1}{2} \left(
	\begin{array}{ccc}
	\frac{1}{\sqrt{2\sqrt{2}+5}} & \frac{1}{\sqrt{2\sqrt{2}+5}} & \frac{(1+\sqrt{2})}{\sqrt{2\sqrt{2}+5}} \\ 
	\frac{1}{\sqrt{2\sqrt{2}+5}} & \frac{1}{\sqrt{2\sqrt{2}+5}} & \frac{-1-\sqrt{2}}{\sqrt{2\sqrt{2}+5}} \\ 
	\frac{1}{\sqrt{2\sqrt{2}+5}} & -\frac{1}{\sqrt{2\sqrt{2}+5}} & \frac{(1+\sqrt{2})}{\sqrt{2\sqrt{2}+5}} \\ 
	\frac{1}{\sqrt{2\sqrt{2}+5}} & -\frac{1}{\sqrt{2\sqrt{2}+5}} & -\frac{(1+
		\sqrt{2})}{\sqrt{2\sqrt{2}+5}} \\ 
	\frac{1}{\sqrt{2\sqrt{2}+5}} & \frac{(1+\sqrt{2})}{\sqrt{2\sqrt{2}+5}} & 
	\frac{1}{\sqrt{2\sqrt{2}+5}} \\ 
	\frac{1}{\sqrt{2\sqrt{2}+5}} & \frac{(1+\sqrt{2})}{\sqrt{2\sqrt{2}+5}} & -
	\frac{1}{\sqrt{2\sqrt{2}+5}} \\ 
	\frac{1}{\sqrt{2\sqrt{2}+5}} & \frac{-1-\sqrt[2]{2}}{\sqrt{2\sqrt{2}+5}} & 
	\frac{1}{\sqrt{2\sqrt{2}+5}} \\ 
	\frac{1}{\sqrt{2\sqrt{2}+5}} & -\frac{(1+\sqrt{2})}{\sqrt{2\sqrt{2}+5}} & 
	-\frac{1}{\sqrt{2\sqrt{2}+5}} \\ 
	\frac{(1+\sqrt{2})}{\sqrt{2\sqrt{2}+5}} & \frac{1}{\sqrt{2\sqrt{2}+5}} & 
	\frac{1}{\sqrt{2\sqrt{2}+5}} \\ 
	\frac{(1+\sqrt{2})}{\sqrt{2\sqrt{2}+5}} & \frac{1}{\sqrt{2\sqrt{2}+5}} & -
	\frac{1}{\sqrt{2\sqrt{2}+5}} \\ 
	\frac{(1+\sqrt{2})}{\sqrt{2\sqrt{2}+5}} & -\frac{1}{\sqrt{2\sqrt{2}+5}} & 
	\frac{1}{\sqrt{2\sqrt{2}+5}} \\ 
	\frac{(1+\sqrt{2})}{\sqrt{2\sqrt{2}+5}} & -\frac{1}{\sqrt{2\sqrt{2}+5}} & 
	-\frac{1}{\sqrt{2\sqrt{2}+5}}
	\end{array}\right).$ Each row $V$ is of norm $\frac{1}{2}$.  Additionally,
      we see the columns are orthogonal and of equal norm.  Therefore,
      V is an isometry and we can conclude that the rows form a
      uniform Parseval frame that is a correlation minimizing frame.
\end{proof}
\begin{thm}
	A correlation minimizing $\left(16,3\right)$-frame is given by
scaling the set of unit norm opposite vertices of the Biscribed
Pentakis Dodecahedron with radius one centered at the origin. This
frame is given by scaling the vectors\\
	 \begin{multline*}
	 \left(0,c_0,\pm c_4\right), \left(c_4,0,\pm c_0\right),\left(c_0,\pm c_4,0\right), 
	\left(c_1,0,\pm c_3\right), \left(c_3,\pm c_1,0\right), \left(0,c_3,\pm c_1\right),\\ \left(0,-c_3,-c_1\right) 
\left(c_2,c_2,c_2\right), \left(c_2,-c_2,-c_2\right), \left(-c_2,c_2,-c_2\right),
\left(-c_2,-c_2,c_2\right).
\end{multline*}
 $c_0= \frac{\sqrt{15}-\sqrt{3}}{6}$, $c_1=
 \frac{\sqrt{10\left(5-\sqrt{5}\right)}}{10}$, $c_2=
 \frac{\sqrt{3}}{3}$,$c_3=
 \frac{\sqrt{10\left(5+\sqrt{5}\right)}}{10}$, and $c_4=
 \frac{\sqrt{15}+\sqrt{3}}{6}$ by the factor $\frac{\sqrt{3}}{4}.$
\end{thm}
\begin{proof}
	Let the columns of $W^*$ be the opposite vertices of the Biscribed Pentakis Dodecahedron with radius one centered at the origin in \cite{DC}.  Set $c_0= \frac{\sqrt{15}-\sqrt{3}}{6}$, $c_1= \frac{\sqrt{10\left(5-\sqrt{5}\right)}}{10}$, $c_2= \frac{\sqrt{3}}{3}$,$c_3= \frac{\sqrt{10\left(5+\sqrt{5}\right)}}{10}$, and 
	$c_4= \frac{\sqrt{15}+\sqrt{3}}{6}$.  It follows that, \\
	$W^*=
	\left(\begin{array}{cccccccccccccccc}
	0   & 0    & c_4 & c_4 & c_0 & c_0  & c_1 & c_1  & c_3 & c_3  & 0   & 0    & c_2 & c_2  & -c_2 & -c_2 \\
	c_0 & c_0  & 0   &  0  & c_4 & -c_4 & 0   & 0    & c_1 & -c_1 & c_3 & -c_3 & c_2 & -c_2 & c_2  & -c_2\\
	c_4 & -c_4 & c_0 & -c_0&  0  & 0    & c_3 & -c_3 & 0   & 0    & c_1 & -c_1 & c_2 & -c_2 & -c_2 & c_2
	
	\end{array}
	\right)$.\\
	By inspection we see that the columns of $W^*$ are unit norm,
        the rows are equal norm
        and the rows are orthogonal. Hence, $V= \frac{\sqrt{3}}{4}W$
        is an isometry and so its rows are a uniform $(16,3)$-frame. Since these vectors are the vertices of the the optimal line packing in \cite{CHS96} they form a $(16,3)$-frame, which is correlation minimizing.
\end{proof}  

\begin{remark} We do not know if the correlation minimizing
  $(12,3)$-frame and $(16,3)$-frame are unique up to frame
  equivalence. This is largely because it is unknown if the
  corresponding arrangements for the optimal packings of $16$ lines
  and $12$ lines, respectively, in $\bb R^3$ are unique.
\end{remark}

\end{document}